\documentclass[12pt]{amsart}
\usepackage{amssymb}
\usepackage{graphicx}
\usepackage{colordvi}
\usepackage{youngtab}
\usepackage{tikz}
\usetikzlibrary{arrows}
\usepackage{tikz-cd}
\usepackage{apacite}

\numberwithin{equation}{section}
\newtheorem{theorem}{Theorem}[section]
\newtheorem{proposition}[theorem]{Proposition}
\newtheorem{lemma}[theorem]{Lemma}
\newtheorem{corollary}[theorem]{Corollary}
\theoremstyle{remark}
\newtheorem{example}[theorem]{Example}
\newtheorem{remark}[theorem]{Remark}
\newtheorem{defn}[theorem]{Definition}

\newcounter{FNC}[page]
\def\fauxfootnote#1{{\addtocounter{FNC}{2}$^\fnsymbol{FNC}$%
     \let\thefootnote\relax\footnotetext{$^\fnsymbol{FNC}$#1}}}
\newcommand{\C}{{\mathbb{C}}}
\newcommand{\N}{{\mathbb{N}}}
\newcommand{\Z}{{\mathbb{Z}}}
\newcommand{\R}{{\mathbb{R}}}

\title{Graphs of Reduced Words  and Some Connections}

\author{Praise Adeyemo}
\address{Department of Mathematics\\
  University of Ibadan\\
   Ibadan, Oyo, Nigeria}
\email{ph.adeyemo@ui.edu.ng, praise.adeyemo13@gmail.com}
\urladdr{http://sci.ui.edu.ng/HPAdeyemo}


\copyrightinfo{}{}

\begin{document}
\begin{abstract}
  The family of graphs of reduced words of a certain subcollection  of permutations in the union $\cup_{n\geq 4}\frak{S}_{n}$ of symmetic groups is investigated. The subcollection is characterised by the hook cycle type $(n-2,1,1)$ with consecutive fixed points. A closed formula for counting the vertices of each member of the family  is given and the vertex-degree polynomials for the graphs with their generating series is realised. Lastly, some isomorphisms of  these graphs with various combinatorial objects are established.    

\end{abstract}

\maketitle
%
\section{Introduction}
\noindent The connection between reduced words of permutations and tableaux dated back to the 1982 paper  of A. Lascoux and Mercel-Paul Schutzenberger [9]. Ever since the connection was found there have been a significant amount of research along this direction. The reader is referred to [3], [5], [7], [11], [12] and [14]. We study the reduced word graphs of the permutations $_{n}w$  of the form
 $$ [n,1,2,3,\dots, n-4, n-2, n-1, n-3], \ {\rm for} \  \ n\geq 4.$$
These permutations are characterised by the hook cycle type $(n-2, 1,1)$ with consecutive fixed points at the positions $n-1$ and $n-2$. The characterisation allows the establishment of a bijection between the set of reduced words of  $_{n}w$ and a certain sub-collection of the set of row-strict tableaux $\tau$ shaped $(n-2,1,1)$ such that the entries of each tableau encode the positions of ascents and descents of every reduced words of $_{n}w$ with respect to a specific  row reading. This gives rise to a closed formula for counting the reduced words of $_{n}w$ in terms of  multinomial coefficients.   It turns out that there is an isomorphism between the graph of reduced words of the permutation  and the Hasse diagram $\mathcal{R}^{_{\tau}w}_{n}$ of the sub-collection of row-strict tableaux $\tau$. It is observed that the set of row-readings of the sub-collection of  row-strict tableaux $\tau$ is precisely the set of Grassmannian permutations with a unique descent at $(n-2)$. In fact, these are exactly all the Grassmannian permutations whose associated partitions  have Young diagrams which fit into the rectangle $\square_{n-2\times 2}$. It well known that  the usual q-binomial coefficient q-counts the set of  such partitions:
$$\left[ n \choose n-2\right]_q =\sum_{\lambda} q^{|\lambda|}$$
where $|\lambda|$ denotes the number partitioned by $\lambda$. This has an interesting implication in geometry in that these partitions index the Schubert varieties of the Grassmannian variety ${\rm Gr}(2,n)$. The observation reveals a fundamental  connection between  the graph $\mathcal{G}_{_{n}w}$ of the reduced words of the permutation $_{n}w$ and the the graph  $\mathcal{G}_{(n-2)\Delta_{2}\cap \Z^{2}}$ of $(n-2)$-fold dilation of the standard 2-simplex in that to every lattice point {\bf a} in $(n-2)\Delta_{2}\cap \Z^{2}$ there is a fitted partition $\lambda$ and a weight ${\bf m}_{\bf a}$ such that the length of the Grassmannian permutation $w(\lambda)$ is ${\bf m}_{\bf a}$. Some necessary backgrounds on the symmetric group $\frak{S}_{n}$ are reviewed and two preliminary results are given in section 2. The bijecton between the vertex set of the graph $\mathcal{G}_{_{n}w}$ of reduced words of the permutation $_{n}w$ and the sub-collection of row-strict tableaux $\tau$ of shape $(n-2,1,1)$ which brings out the closed formula for the order of the graph is established in section 3. The symmetry in the structure of the graph $\mathcal{G}_{_{n}w}$ is exploited in the same section to give the vertex-degree polynomials ${\rm P}_{\mathcal{G}_{_{n}w}}(d)$ for $\mathcal{G}_{_{n}w}, \  n\geq 4$ and their generating series. In section 4 we give a poset of the sub-collection of row-strict tableaux $\tau$ and establish an isomorphism between the graph $\mathcal{G}_{_{n}w}$ and the Hasse diagram of the row-strict tableaux. We show that the Hasse diagram is graded by identifying it with the  sub-poset of Grassmannian permutations  induced by  the strong Bruhat graph of the symmetric group $\mathcal{S}_{n}$. In section 5  we give a construction of poset of the lattice points ${(n-2)\Delta_{2}\cap \Z^{2}}$  of $(n-2)$-fold dilation of the standard $2$-simplex using lexicographic ordering. The Hasse diagram of the poset is graded such that  the rank polynomial  is a refinement of the Ehrhart polynomial of the simplex. The graph of  lattice points is realised as the Hasse diagram of the poset.  
\section{Reduced Decompositions}
\noindent It is well known that the symmetric group $\mathfrak{S}_{n}$ on the  set $[n]:=\{1,\ldots, n\}$ is finitely presented. That is,
\begin{equation}
\mathfrak{S}_{n}= \langle s_1,\dots, s_{n-1} :  s_{i}^{2} =e,\ s_{i}s_{j} = s_{j}s_{i} , \ s_{i}s_{i+1}s_{i} = s_{i+1}s_{i}s_{i+1} \rangle
\end{equation}
The generators $s_1, s_2,\dots, s_{n-1}$ are simple transpositions each of which is an involution. The relations $s_{i}s_{j} = s_{j}s_{i}$  for  $\mid i-j\mid > 1$ and $s_{i}s_{i+1}s_{i} = s_{i+1}s_{i}s_{i+1}$ \  for  $1\leq i\leq n-2$ are called commutation and braid respectively. We shall adopt one-line notation for the elements in $\mathfrak{S}_{n}$.
\noindent The length $\ell(w)$ of the permutation $w\in\mathfrak{S}_{n}$ is given by the number of inversions in $w$,
\begin{equation}
\ell(w) = |\{(i,j)\in [n]^{2} : w(i)> w(j), 1\leq i<j\leq n\}|
\end{equation}
and the length generating function of the symmetric group $\mathfrak{S}_{n}$ is given by
\begin{equation}
G_{\mathfrak{S}_{n}} (q)= \prod^{n}_{k=1} \frac{q^{k}-1}{q-1}. 
\end{equation}
 \noindent  The symmetric group $\mathfrak{S}_{n}$  is a poset under the Bruhat order, that is, for  $v, v'\in \mathfrak{S}_{n}$, $v < v'$ if $v'$ is realised from $v$ by interchanging $v_{i}$ and  $v_{j}$ where $i<j$ and $v_{j} = v_{i}+1$. This defines strong Bruhat order on $\mathfrak{S}_{n}$, the weak Bruhat order is given by: for  $v, v'\in \mathfrak{S}_{n}$, $v < v'$ if $v'$ is realised from $v$ by interchanging $v_{i}$ and  $v_{i+1}$ where  $v_{i} < v_{i+1}$.  Bruhat order encodes some geometric interpretations as it describes the inclusion ordering of Schubert varieties of certain homogeneous spaces such as flag manifolds and Grassmannians.   The  poset of $\mathfrak{S}_{n}$  with respect to the Bruhat order is graded and the length generating function described in the equation $(2.3)$ is its rank function [5]. For instance, the Hasse diagram of $\mathfrak{S}_{4}$ with respect to the Bruhat order is given in Fig. 1 and its rank function is 
$$G_{\mathfrak{S}_{4}} (q)= q^6+ 3q^5 + 5q^4 + 6q^3+ 5q^2 + 3q + 1 . $$
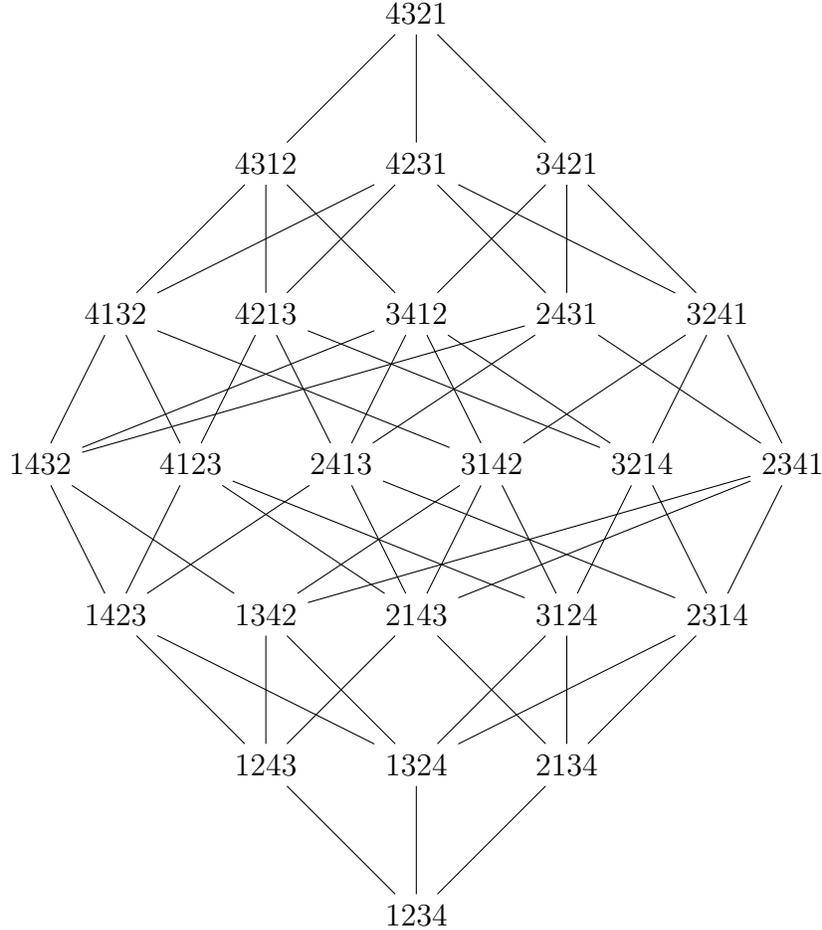
\begin{figure}[!hbt]
	\begin{center}
		\begin{tikzpicture}
		\tikzset{vertex/.style = {shape=circle,draw,minimum size=1.5em}}
		\tikzset{edge/.style = {->,> = latex'}}
		\node (a) at (0,0.5) {$1234$};
		\node (b) at (-2,2.5) {$1243$};
		\node (c) at (0,2.5) {$1324$};
		\node (d) at (2,2.5) {$2134$};
		\node (e) at (-4,4.5) {$1423$};
		\node (f) at (-2,4.5) {$1342$};
		\node (g) at (0,4.5) {$2143$};
		\node (h) at (2,4.5) {$3124$};
		\node (i) at (4,4.5) {$2314$};
		\node (j) at (-5,6.5) {$1432$};
		\node (k) at (-3,6.5) {$4123$};
		\node (l) at (-1,6.5) {$2413$};
		\node (m) at (1,6.5) {$3142$};
		\node (n) at (3,6.5) {$3214$};
		\node (o) at (5,6.5) {$2341$};
		\node (p) at (-4,8.5) {$4132$};
		\node (q) at (-2,8.5) {$4213$};
		\node (r) at (0,8.5) {$3412$};
		\node (s) at (2,8.5) {$2431$};
		\node (t) at (4,8.5) {$3241$};
		\node (u) at (-2,10.5) {$4312$};
		\node (v) at (0,10.5) {$4231$};
		\node (w) at (2,10.5) {$3421$};
		\node (x) at (0,12.5) {$4321$};
		\draw (b) to (a);
		\draw (c) to (a);
		\draw (d) to (a);
		\draw (e) to (b);
		\draw (f) to (b);
		\draw (g) to (b);
		\draw (e) to (c);
		\draw (f) to (c);
		\draw (h) to (c) ;
		\draw (i) to (c) ;
		\draw (g) to (d) ;
		\draw (h) to (d) ;
		\draw (i) to (d) ;
		\draw (j) to (e) ;
		\draw (k) to (e) ;
		\draw (l) to (e) ;
		\draw (j) to (f) ;
		\draw (m) to (f) ;
		\draw (o) to (f) ;
		\draw (k) to (g) ;
		\draw (l) to (g) ;
		\draw (m) to (g) ;
		\draw (o) to (g) ;
		\draw (k) to (h) ;
		\draw (m) to (h) ;
		\draw (n) to (h) ;
		\draw (l) to (i) ;
		\draw (n) to (i) ;
		\draw (o) to (i) ;
		\draw (p) to (j) ;
		\draw (r) to (j) ;
		\draw (s) to (j) ;
		\draw (p) to (k) ;
		\draw (q) to (k) ;
		\draw (q) to (l) ;
		\draw (r) to (l) ;
		\draw (s) to (l) ;
		\draw (p) to (m) ;
		\draw (r) to (m) ;
		\draw (t) to (m) ;
		\draw (q) to (n) ;
		\draw (r) to (n) ;
		\draw (t) to (n) ;
		\draw (s) to (o) ;
		\draw (t) to (o) ;
		\draw (u) to (p) ;
		\draw (v) to (p) ;
		\draw (u) to (q) ;
		\draw (v) to (q) ;
		\draw (u) to (r) ;
		\draw (w) to (r) ;
		\draw (w) to (s) ;
		\draw (v) to (s) ;
		\draw (v) to (t) ;
		\draw (w) to (t) ;
		\draw (x) to (u) ;
		\draw (x) to (v) ;
		\draw (x) to (w) ;
		\end{tikzpicture}
	\end{center}
	\caption{Hasse diagram of strong Bruhat order on  $\frak{S}_{4}$ .}
\end{figure}
 Recall that the length of $w\in\frak{S}_{n}$ is invariant  under inversion $w\mapsto w^{-1}$ and conjugation $w\mapsto w_{0}ww_{0}$, where $w_{0}$ is the longest permutation in $\frak{S}_{n}$, that is, $w_{0}:= [n,n-1,n-2,\dots, 3,2,1]$. Therefore, $\ell(w)=\ell(w^{-1})=\ell(w_{0}ww_{0})=\ell(w_{0}w^{-1}w_{0})$. The descent set Des($w$) of $w$ is given by Des($w$)=$\{i : w_{i}> w_{i+1}\}$.  A minimal length for $w$ as a finite product of simple transpositions $s_{i}$ is said to be reduced, that is, $w = s_{a_1}s_{a_2}\cdots s_{p}$, where $p =\ell(w)$. We say that the sequence {\bf a}= $a_{1}a_{2}\cdots a_p$ is a reduced word for $w$. The descent set Des({\bf a}) of the reduced word {\bf a}=$a_{1}a_{2}\cdots a_p$ is given by Des({\bf a})=$\{j : a_{j}> a_{j+1}\}$, likewise the ascent set Asc({\bf a}) of {\bf a}  is given by   Asc({\bf a})=$\{j : a_{j} < a_{j+1}\}$.  The reduced decomposition of $w$ is not unique in that another reduced decomposition of $w$ can be realised by applying either braid or commutation relation.  This observation gives rise to the graph $\mathcal{G}_{w}$ of reduced words of the permutation $w$ in which the reduced words of $w$ constitute the vertices of the graph such that there is an edge between any pair of reduced words if there is either a braid or commutation relation between them. We denote the set of reduced words for $w\in \frak{S}_{n}$ by $R(w)$  and its cardinality by $r(w)$.  For instance, if $w = 35124$ then
\begin{center}
$R(w)=\{42312, 24312, 42132, 24132,21432 \} \ \ {\rm and} \ \ r(w)=5$
\end{center}
The graph $\mathcal{G}_{35124}$  of the reduced words of the permutation $w=35124$  is given below

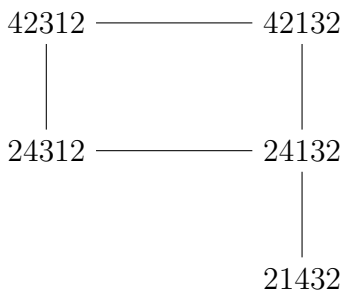
\begin{figure}[!hbt]
	\begin{center}
	  \begin{tikzpicture}[x= 1.7cm, y=1.7cm]
	   \tikzset{vertex/.style = {shape=circle,draw,minimum size=10.5em}}
	    \tikzset{edge/.style = {-,> = latex'}}
		\node (a) at (0,3) {$42312$};
		\node (b) at (0,2) {$24312$};
		\node (c) at (2,2) {$24132$};
		\node (d) at (2,3) {$42132$};
		\node (e) at (2,1) {$21432$};
		\draw (e)--(c)--(b)--(a)--(d)--(c);
		\end{tikzpicture}
	\end{center}
	\caption{The graph $\mathcal{G}_{35124}$ of the reduced words of $w=35124$}
\end{figure}    
\noindent If $w=[6,5,4,2,3,1]$ then $r(w) =64,064$. Given any permutation $w\in\mathfrak{S}_{n}$, what is $r(w)$?. The question is interesting due to fact that $r(w)$ encodes some combinatorial and geometric interpretations for certain class of permutations. In order to compute $r(w)$, Stanley [12] introduced a certain function $F_{w}$ called Stanley symmetric function. It turns out that the coefficient of the square free monomial in $F_{w}$ is number of the reduced words of $w$. This is given by 
\begin{equation}
r(w) = \sum_{\lambda\vdash \ell(w)} a_{\lambda(w)} f^{\lambda}
\end{equation}
Unfortunately, there is no general closed formula for the computation of $r(w)$ for any given $w\in\mathfrak{S}_{n}$.  However, for $w_{0}$,  the longest permutation in $\mathfrak{S}_{n}$, Stanley [12] gave a closed formula for $r(w_{0})$ as 
\begin{equation}
r(w_{0})= \frac{{n\choose 2}!}{1^{n-1}3^{n-2}5^{n-3}\cdots (2n-3)^{1}}
\end{equation}
This is precisely the number $f^{n-1,n-2,\dots, 3,2,1}$ of standard tableaux of shape $\lambda =(n-1,n-2,\dots, 3,2,1)$. The goal of this paper is to study the graphs of reduced words of the permutations $_{n}w$ in the family $\mathcal{W}\subset\cup_{n\geq 4}\mathfrak{S}_{4}$. These permutations are of the form:
\begin{equation}
 _{n}w = [n,1,2,3,\dots, n-4, n-2, n-1, n-3], \ {\rm for} \  \ n\geq 4
\end{equation}
For instance,  $_{4}w = 4231, \  _{5}w = 51342, \  _{6}w = 612453$, etc.
We care about the graphs of reduced words of these permutations   due to their interesting symmetry and  deep connection with the Schubert varieties of Grassmannian Gr$(2,n)$, of 2-dimensional subspaces in an $n$-dimensional vector space over the complex field. 
\begin{proposition}
Let $_{n}w\in \mathcal{W}$. Then $_{n}w$ has the following properties
\begin{enumerate}
\item[i.] The length $_{n}w$  is $n+1$.
\item[ii.]  The cycle type of $_{n}w$ is  $(n-2,1,1)$
\item[iii.] The  consecutive fixed points of $_{n}w$ occur at  the positions $n-2$ and $n-1$
\item[iv.] $_{n}w$ has only 2 ascents.
\end{enumerate}
\end{proposition}
\begin{proof}
 Following the definition of length and looking at the form of $_{n}w$ in $(1.6)$ there are exactly $n-1$ inversions associated with the value $n$ since every value to the right of $n$ is strictly less that $n$.  Also each of the value $n-2$ and $n-1$ has exactly one inversion associated to it. The values $1,2,3,\dots,n-4$ are in increasing order so there is no inversion. Therefore, $_{n}w^{n}$ has $n+1$ inversions.\\
 (ii) $_{n}w$ is an $(n-2)$-cycle of the form  $w=(1,  _{n}w(1), _{n}w^{2}(1),\cdots ,_{n}w^{n-3}(1))$, so it is of the type $(n-2,1,1)$.\\
 (iii) From (ii), $_{n}w$ has two fixed points and these fixed points at the positions $n-2$ and $n-1$.\\
 (iv.) From (1.6) the descents in $_{n}w$ occur at the first and last but one positions. 
\end{proof}
\begin{proposition}
Let $_{n}w\in \mathcal{W}$ such that $\sigma\in \frak{S}_{n}$ is its reverse, that is, $\sigma_{i}= _{n}w_{n+1-i}$ for $1\leq i\leq n$. Then $\ell(_{n}w)+\ell(\sigma) = {n\choose 2}$.
\end{proposition}
\begin{proof}
By definition the permutation $\sigma$ is of the form $[n-3,n-1,n-2,n-4,\dots, 3,2,1,n]$, so, the value $n-3$ has $n-4$ inversions associated to it, the value $n-1$ has $n-3$ associated inversions while the value $n-2$ gives $n-4$ inversions. There are ${r-4\choose 2}$ inversions within the  values $n-4,\dots,3,2,1$ being in decreasing order. Therefore, $\ell(\sigma)=\frac{n^{2}-3n-2}{2}$. Recall that $\ell(_{n}w)=n+1$ so that $\ell(_{n}w)+\ell(\sigma) = {n\choose 2}$.
\end{proof}

\section{Graph of Redced Words of $_{n}w$}
\noindent In this section we study the graphs of reduced words of the permutations $_{n}w$ for $n\geq 4$. The reduced words constitute the vertices of the graphs in question while the edges are given either by commutation or braid relation. For instance, if $n = 4$, then the set $R(_{4}w)$ of reduced words of the permutation $_{4}w = 4231$ is given by
\begin{equation}
R(_{4}w) =\{ 32123, 31213, 13213, 31231, 13231, 12321 \}.
\end{equation}
Its graph, denoted by $\mathcal{G}_{_{4}w}$ is given in Figure 3 . The first goal is to give a closed formula for the order of the graph $\mathcal{G}_{_{n}w}$ for any $n$, that is, the number of the vertices of $\mathcal{G}_{_{n}w}$. This is the cardinality $r(_{n}w)$ for the set $R(_{n}w)$. The following proposition is very important in order to count the vertices of $\mathcal{G}_{_{n}w}$.

\begin{proposition}
Let ${\bf a}\in R(_{n}w)$ then {\bf a} has exactly 2 ascents and $n-2$ descents
\end{proposition}
\begin{proof}
Notice that  by induction there only two degree one vertices in  the graph $\mathcal{G}_{_{n}w}$ of the reduced words of the permutation $_{n}w$ with 2 ascents each, these are  $(n-1)(n-2)(n-3)\cdots 321(n-2)(n-1)$ and $(n-3)(n-2)(n-1)(n-2)\cdots 321$. Since these vertices are respectively located at the top  and bottom of the graph, so every reduced word of $_{n}w$ lies on the paths  between them. The number of ascents or descents is invariant under  commutation and braid relations.
\end{proof}
 \noindent We now give a combinatorial object $\mathcal{C}_{n-2,1,1}$ called the recorder which shall simultaneously keep track of both descent and ascent positions of every reduced word of the permutation $_{n}w$. We recall some basic facts: A partition $\lambda$ of $n\in\N$ denoted $\lambda\vdash n$  is a list $\lambda=(\lambda_1\geq\lambda_2\geq\cdots\geq\lambda_k)$ such that $\lambda_1+\lambda_2+\dots +\lambda_k= |\lambda|=n$. To each partition $\lambda\vdash n$ there is an associated  diagram $Y(\lambda)$ called the Young diagram of shape $\lambda$ consisting of $|\lambda|$ boxes having $k$ left justified rows with row $i$ containing $\lambda_i$ boxes for $1\leq i\leq k$. For example if $\lambda = 5,4,3,2,1$ then its Young diagram is given by
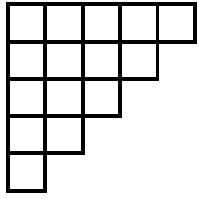
\[\Ylinethick{1pt}\yng(5,4,3,2,1)\]
\noindent There are many ways of describing a Young tableau associated with  the Young diagram of shape $\lambda\vdash n$. The reader is referred to [6], [10], [13]. We care about the row-strict tableaux which are given by the filling  the boxes of the Young diagram of shape $\lambda\vdash n$  with numbers from the set $[n]:=\{1,\dots,n\}$ such that the numbers in the rows strictly increase from left to right.  Let ${\rm RST}_{n}(\lambda)$ denote the set of all row-strict  tableaux of shape $\lambda=(\lambda_{1},\lambda_{2},\dots,\lambda_{k})\vdash n$. The cardinality $|{\rm RST}_{n}(\lambda)|$ of the set  given by the multinomial coefficient:
\begin{equation}
|{\rm RST}_{n}(\lambda)|={r\choose{\lambda_{1},\dots,\lambda_{k}}}.
\end{equation}
The focus is on the sub-collection of row-strict tableaux of the hook shape $(n-2,1,1)\vdash n$, where $n\geq 4$.
\begin{defn}
A row-strict tableau $\tau\in {\rm RST}_{n}(n-2, 1^2)$ is said to be a recording row strict tableau if the numbers in  the last two boxes $\Ylinethick{0.5pt}\yng(1,1)$ of the first column strictly decrease downward.
\end{defn}
\noindent For instance, $\tau = \begin{array}{l}{\young(23,1,4)}\end{array}$ is not a recording row strict tableau, while $\eta=\begin{array}{l}{ \young(23,4,1)}\end{array}$ is.

 \begin{figure}[!hbt]\label{fig_r=5}
          \begin{center}
	  \begin{tikzpicture}[x= 1.7cm, y=1.7cm]
	   \tikzset{vertex/.style = {shape=circle,draw,minimum size=10.5em}}
	    \tikzset{edge/.style = {-,> = latex'}}
                \node (a) at (-1,0.5) {${\color{red}12321}$};
    		\node (b) at (0,1) {${\color{red}13231}$};
		\node (c) at (1,1.5) {${\color{red}13213}$};
		\node (h) at (-1,2.5) {${\color{red}32123}$};
		\node (i) at (0,2) {${\color{red}31213}$};
		\node (j) at (-1,1.5) {${\color{red}31231}$};
		\draw (b)--(c);
		\draw (h)--(i);
		\draw (b)--(j)--(i)--(c); 
		\draw (a)--(b);
\end{tikzpicture}
	\end{center}
	\caption{The graph $\mathcal{G}_{_{4}w}$ of the reduced words of $_{4}w=4231$}		
\end{figure}
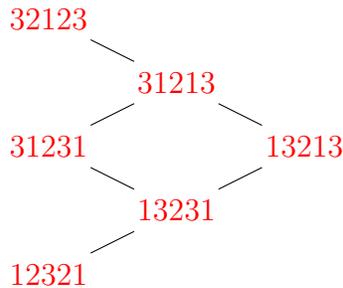

\begin{proposition}
Let $\mathcal{C}_{n-2, 1,1}$ be the set of recording  row-strict tableaux in ${\rm RST}_{n}(n-2, 1,1)$. Then $|\mathcal{C}_{n-2, 1,1}|= \frac{1}{2} {n\choose{n-2, \ 1,\ 1}}$. 
\end{proposition}
\begin{proof}
The row-strict filling with  numbers from $[n]:=\{1,\dots,n\}$ for the shape $(n-2,1,1)\vdash r$ requires ${{n}\choose n-2}$ choices for the first row, ${{2}\choose 1}$ for the second row and ${1\choose 1}$ for the last row. Therefore, $\mathcal{C}_{n-2, 1,1}|={\frac{1}{2}}{r\choose r-2}\cdot {2\choose 1}\cdot  {1\choose 1 }$, since $\mathcal{C}_{n-2, 1,1}$, by definition, requires numbers in  the last two boxes of the first column to strictly decrease downward.
\end{proof}

\begin{example}
 There are 20 row-strict tableaux associated to the Young diagram of the shape $\lambda=(3,1,1)$, that is,

\[RST(3,1,1)=\left\lbrace \begin{array}{l} \vspace{10pt} \vcenter{\hbox{\young(123,4,5),\, \young(124,3,5),\, \young(125,3,4),\, \young(134,2,5),\, \young(135,2,4),\, \young(145,2,3)}}\\ \vspace{10pt}\vcenter{\hbox{\young(234,1,5),\, \young(235,1,4),\, \young(245,1,3),\, \young(345,1,2),\, \young(123,5,4),\, \young(124,5,3)}}\\  \vspace{10pt}\vcenter{\hbox{\young(125,4,3),\, \young(134,5,2),\, \young(135,4,2),\, \young(145,3,2),\, \young(234,5,1),\, \young(235,4,1)}} \\ \vcenter{\hbox{\young(245,3,1),\, \young(345,2,1)}}\end{array}\right\rbrace .\] 
Exactly 10 of them are the elements of $\mathcal{C}_{3, 1,1}$
\end{example}
\begin{theorem}
Let $\mathcal{G}_{_{n}w}$ be the graph of reduced words of the permutation $_{n}w$. Then the order $r(_{n} w)$ of $\mathcal{G}_{_{n}w}$ is given by the half of the multinomial coefficient;
$$r(_{n}w) = \frac{1}{2} {n\choose{n-2, \ 1,\ 1}}.$$
\end{theorem}  
\begin{proof} It follows from the  bijection between the set $R(_{n} w)$ of the reduced words of the permutation $_{n}w$ and the set $\mathcal{C}_{n-2, 1,1}$  of recording  row-strict tableaux. The reduced words of  $R(_{n}w)$ have  $n-2$ descents and 2 acents each by Proposition 2.1. So to every  ${\bf a}\in R(_{n}w)$ there is a unique recording row strict tableau $\tau\in\mathcal{C}_{n-2, 1,1}$ such that the entries in the first row record the positions of descents in {\bf a} while the last two rows capture the location of ascents in the same reduced words..
\end{proof}
\begin{example}
The fact that $r(_{5}w)=10$ is established from the bijection between the vertices of the graph $\mathcal{G}_{_{5}w}$ of reduced words of the permutation $_{5}w=51342$ and the set $\mathcal{C}_{3, 1,1}$ of recording row strict tableaux  shaped 3,1,1 in which the tableaux are arranged in such a manner that resembles the structure of the graph, see Figure 4.
\end{example}
\noindent  We shall describe in the next section a poset $(\mathcal{C}_{n-2, 1,1},   \leq)$ of  recording row strict tableaux whose Hasse diagram  preserves the structure of the corresponding graph of reduced words.   

\begin{figure}[!hbt]\label{fig_r=5}
	  \begin{tikzpicture}[x= 1.7cm, y=1.7cm]
	   \tikzset{vertex/.style = {shape=circle,draw,minimum size=10.5em}}
	    \tikzset{edge/.style = {-,> = latex'}}
                \node (a) at (-1,0.5) {${\color{red}234321}$};
    		\node (b) at (0,1) {${\color{red}243421}$};
		\node (c) at (1,1.5) {${\color{red}243241}$};
		\node (d) at (2,2) {$243214$};
		\node (e) at (1,2.5) {${\color{blue}423214}$};
		\node (f) at (0,3) {${\color{blue}432314}$};
		\node (g) at (-1,3.5) {${\color{blue}432134}$};
		\node (h) at (-1,2.5) {${\color{red}432341}$};
		\node (i) at (0,2) {${\color{red}423241}$};
		\node (j) at (-1,1.5) {${\color{red}423421}$};
		\draw (b)--(c);
		\draw (h)--(i);
		\draw (e)--(f)--(g);
		\draw (b)--(j)--(i)--(c); 
		\draw (a)--(b);
		 \draw (h)--(f);
		 \draw (i)--(e);
		 \draw (c)--(d);
		 \draw (e)--(d);	
		 
		 \node (k) at (5,0.5) {${\color{red}\young(345,2,1)}$};
    		\node (l) at (6,1) {${\color{red}\young(245,3,1)}$};
		\node (m) at (7,1.5) {${\color{red}\young(235,4,1)}$};
		\node (n) at (8,2) {$\young(234,5,1)$};
		\node (o) at (7,2.5) {${\color{blue}\young(134,5,2)}$};
		\node (p) at (6,3) {${\color{blue}\young(124,5,3)}$};
		\node (q) at (5,3.5) {${\color{blue}\young(123,5,4)}$};
		\node (r) at (5,2.5) {${\color{red}\young(125,4,3)}$};
		\node (s) at (6,2) {${\color{red}\young(135,4,2)}$};
		\node (t) at (5,1.5) {${\color{red}\young(145,3,2)}$};
		\draw (l)--(m);
		\draw (r)--(s);
		\draw (o)--(p)--(q);
		\draw (l)--(t)--(s)--(m); 
		\draw (k)--(l);
		 \draw (r)--(p);
		 \draw (s)--(o);
		 \draw (m)--(n);
		 \draw (o)--(n);
		\end{tikzpicture}	
	\caption{The bijection between the graph $\mathcal{G}_{51342}$ of reduced words of the permutation $_{5}w$ and the set $\mathcal{C}_{3, 1,1}$ of recording row strict tableaux of the shape 3,1,1 arranged in a corresponding manner.}		
\end{figure}
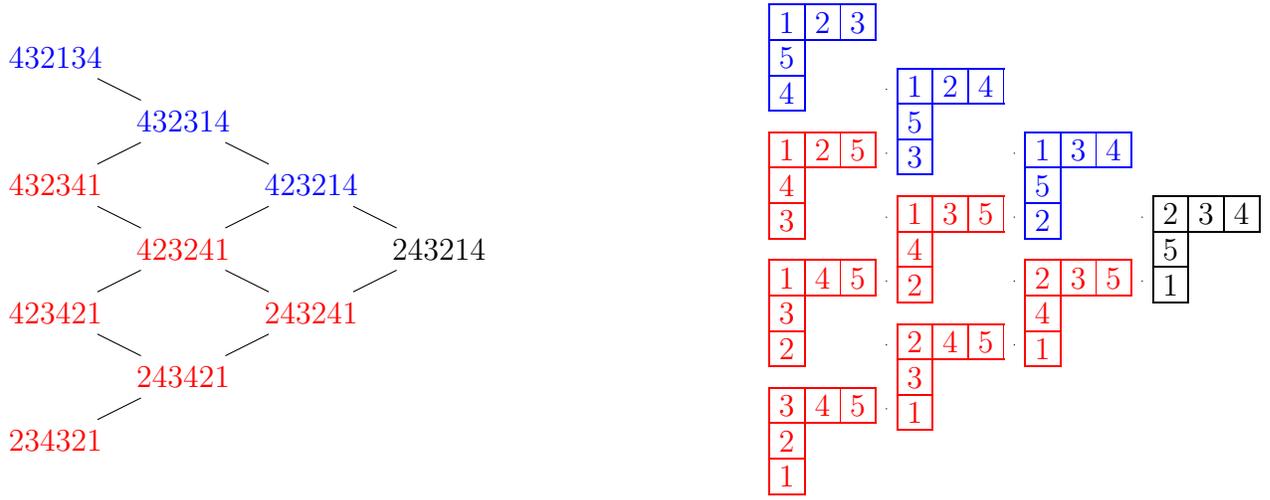
\begin{corollary}
Let $_{n}w\in \mathcal{W}$ be such that $\sigma\in \frak{S}_{n}$ is its reverse, that is, $\sigma_{i} = _{n}w_{n+1-i}$. Then 
$$r(_{n}w) = \ell(_{n}w)+\ell(\sigma).$$
\end{corollary}
\begin{proof}
Notice that $$\frac{1}{2} {n\choose{n-2, \ 1,\ 1}}={n\choose 2}$$
\end{proof}
\begin{example}
Given $_{5}w=51342$, its reverse  $\sigma=24315$, so $\ell(51342)+\ell(24315)=6 + 4$, so there are precisely 10 reduced words for the permutation $_{5}w$.   
\end{example}

\noindent The symmetry of  the graph $\mathcal{G}_{_{n}w}$ with respect to the regular-triangle orientation displayed in the Fig.4  gives a convenient way of keeping track of the degree of every vertex. Notice that the graph is symmetric along the horizontal line through the degree-two corner vertex  $(n-3)(n-1)\cdots 321(n-1)$. The degree of a vertex in $\mathcal{G}_{_{n}w}$ is determined by the number of all possible moves (braid and commutation) present in the reduced word which constitutes the vertex. For instance, the vertex $423241$ in $\mathcal{G}_{_{5}w}$ is of degree 4 since it has  3 commutations and  one braid. 
\begin{lemma}
Let $v$ be a vertex of the graph $\mathcal{G}_{_{n}w}$ of the reduced words of $_{n}w$. Then the degree of $v$ is at most $4$.
\end{lemma}
\begin{proof}
Changing the orientation of  $\mathcal{G}_{_{n}w}$ to a regular-triangle, the vertices of $\mathcal{G}_{_{n}w}$ can be divided into two, those that lie on  the boundary  of the  triangle and the interior ones. The boundary vertices are at most of degree 3 since they lie within [$(n-1)(n-2)(n-3)\cdots 321(n-2)(n-1)$ and $(n-3)(n-2)(n-1)(n-2)\cdots 321$],    [$(n-3)(n-1)\cdots 321(n-1)$ and  $(n-3)(n-2)(n-1)(n-2)\cdots 321$] and  [ $(n-3)(n-1)\cdots 321(n-1)$ and $(n-1)(n-2)(n-3)\cdots 321(n-2)(n-1)$]. The interior ones are of degree 4, since they all have $n-1$ commutation moves and exactly one braid move.
\end{proof}
\begin{remark}
The degree bound is  very sharp. In fact, every graph in the family $\mathcal{W}$ except $\mathcal{G}_{_{4}w}$  attains the bound. We now give the distribution of  degrees of the vertices of $\mathcal{G}_{_{n}w}$  in what follows.
\end{remark}
\begin{figure}[htbp]
	\begin{tikzpicture}[x= .6cm, y=.6cm]
	\tikzset{vertex/.style = {shape=circle,draw,minimum size=1.5em}}
	\tikzset{edge/.style = {-,> = latex'}}
	
	\draw[fill=black]	(8,-1) circle (1.8pt);
	\draw[fill=black]	(6,-2) circle (1.8pt);
	\draw[fill=black]	(4,-3) circle (1.8pt);
	\draw[fill=black]	(6,-4) circle (1.8pt);
	\draw[fill=black]	(8,-5) circle (1.8pt);
	\draw[fill=black]	(8,-3) circle (1.8pt);
	
	\draw[fill=black]	(2,-4) circle (1.8pt);
	\draw[fill=black]	(4,-5) circle (1.8pt);
	\draw[fill=black]	(6,-6) circle (1.8pt);
	\draw[fill=black]	(8,-7) circle (1.8pt);
	
	\draw[fill=black]	(10,-4) circle (1.8pt);
	\draw[fill=black]	(14,-4) circle (1.8pt);
	
	\draw[fill=black]	(8.5,-.75) circle (.5pt);
	\draw[fill=black]	(9,-.5) circle (.5pt);
	\draw[fill=black]	(9.5,-.25) circle (.5pt);
	
	\draw[fill=black]	(10,0) circle (1.8pt);
	\draw[fill=black]	(12,1) circle (1.8pt);
	\draw[fill=black]	(14,2) circle (1.8pt);
	
	\draw[fill=black]	(9,-2.5) circle (.5pt);
	\draw[fill=black]	(10,-2) circle (.5pt);
	\draw[fill=black]	(11,-1.5) circle (.5pt);
	
	\draw[fill=black]	(12,-1) circle (1.8pt);
	\draw[fill=black]	(14,0) circle (1.8pt);

	\draw[fill=black]	(11,-3.5) circle (.5pt);
	\draw[fill=black]	(12,-3) circle (.5pt);
	\draw[fill=black]	(13,-2.5) circle (.5pt);
	
	\draw[fill=black]	(14,-2) circle (1.8pt);
	
	\draw[fill=black]	(14,-2.5) circle (.5pt);
	\draw[fill=black]	(14,-3) circle (.5pt);
	\draw[fill=black]	(14,-3.5) circle (.5pt);
	
	\draw[fill=black]	(10,-8) circle (1.8pt);
	\draw[fill=black]	(12,-9) circle (1.8pt);
	\draw[fill=black]	(14,-10) circle (1.8pt);
		
	\draw[fill=black]	(8.5,-7.25) circle (.5pt);
	\draw[fill=black]	(9,-7.5) circle (.5pt);
	\draw[fill=black]	(9.5,-7.75) circle (.5pt);
	
	\draw[fill=black]	(12,-7) circle (1.8pt);
	\draw[fill=black]	(14,-8) circle (1.8pt);
	
	\draw[fill=black]	(9,-5.5) circle (.5pt);
	\draw[fill=black]	(10,-6) circle (.5pt);
	\draw[fill=black]	(11,-6.5) circle (.5pt);
	
	\draw[fill=black]	(14,-6) circle (1.8pt);
	
	\draw[fill=black]	(12,-5) circle (1.8pt);

	\node (A) at (-1.5,-3.7) {$(n-3)(n-1)\cdots 321(n-1)$};
	\node (B) at (15,-10.5) {$(n-3)(n-2)(n-1)(n-2)\cdots 321$};
	\node (C) at (15,2.5) {$(n-1)(n-2)(n-3)\cdots 321(n-2)(n-1)$};
	\node (D) at (18,-4) {$\text{line of symmetry}$};

	\draw[thick] (8,-1)--(6,-2)--(4,-3)--(6,-4)--(8,-5);
	\draw[thick] (6,-2)--(8,-3)--(6,-4);
	\draw[thick] (2,-4)--(4,-5)--(6,-6)--(8,-5);
	\draw[thick] (2,-4)--(4,-3);
	\draw[thick] (4,-5)--(6,-4);
	\draw[thick] (6,-6)--(8,-7);
	
	\draw[thick] (8,-3)--(10,-4)--(8,-5);
	\draw[thick] (12,1)--(14,0)--(12,-1)--(10,0)--(12,1)--(14,2);
	\draw[thick] (14,-2)--(12,-1);
	
	\draw[thick] (12,-9)--(14,-8)--(12,-7)--(10,-8)--(12,-9)--(14,-10);
	
	\draw[thick] (14,-6)--(12,-7);
	
	\draw[thick] (14,-4)--(12,-5)--(14,-6);
	
	\draw[dashed] (0,-4)--(16,-4);
	
	\end{tikzpicture}
	\caption{The graph $\mathcal{G}_{n^w}$ of reduced words of the permutation $n^w$.}
\end{figure}
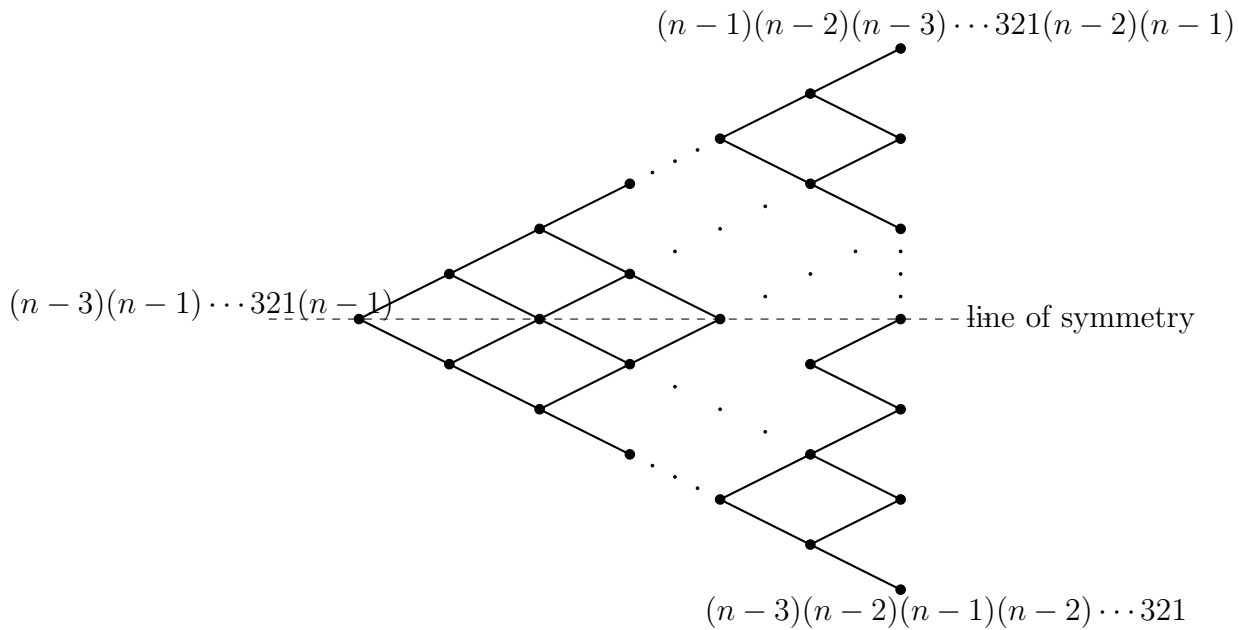

\begin{theorem}
Let $\mathcal{G}_{_{n}w}$ be the graph of reduced words of the permutation $_{n}w$. Then the vertex-degree polynomial  is given by
$${\rm P}_{\mathcal{G}_{_{n}w}}(d)= 2d +(n-2)d^2 + (2n-6)d^{3}+{n-3\choose 2}d^4$$
with $$\sum_{v\in\mathcal{G}_{_{n}w}}{\rm deg}v =4{n-1\choose 2}.$$
Furthermore, there are precisely ${n-2\choose 2}$ 4 cycles in  $\mathcal{G}_{_{n}w}$.
\end{theorem}
\begin{proof}
 Considering the regular-triangle orientation of $\mathcal{G}_{_{n}w}$ with the three corner vertices: [$(n-1)(n-2)(n-3)\cdots 321(n-2)(n-1)$], [$(n-3)(n-2)(n-1)(n-2)\cdots 321$] and [$(n-3)(n-1)\cdots 321(n-1)$], there are only 2 vertices in $\mathcal{G}_{_{n}w}$ which are of degree one. These are precisely the corner vertices  $(n-1)(n-2)(n-3)\cdots 321(n-2)(n-1)$ and $(n-3)(n-2)(n-1)(n-2)\cdots 321$.  The third corner vertex $(n-3)(n-1)\cdots 321(n-1)$ is of degree 2 and the graph  $\mathcal{G}_{_{n}w}$ is symmetric along the horizontal line through the corner vertex, so, the shortest path between the degree-one corner vertex [$(n-1)(n-2)(n-3)\cdots 321(n-2)(n-1)$] and the degree-two corner vertex $(n-3)(n-1)\cdots 321(n-1)$ consists of $n-3$ degree-three inner vertices likewise on the other side of the line of symmetry. These account for all the $2n-6$ degree-three vertices of $\mathcal{G}_{_{n}w}$. Also from the regular-triangle orientation of $\mathcal{G}_{_{n}w}$, there are exactly $n-3$ degree-two inner vertices which lie  on the boundary between [$(n-1)(n-2)(n-3)\cdots 321(n-2)(n-1)$] and [$(n-3)(n-2)(n-1)(n-2)\cdots 321$] making the total number of degree-two vertices in $\mathcal{G}_{_{n}w}$ $n-2$. Since there are ${n\choose 2}$ vertices in $\mathcal{G}_{_{n}w}$ and the degree of every vertex is at most four. Therefore, there are exactly ${n-3\choose 2}$ vertices which are of degree four. Next, we show that the number of edges in $\mathcal{G}_{_{n}w}$ is $2{n-1\choose 2}$. Notice that $\mathcal{G}_{_{n}w}$ is a bipartite graph and the fact that the order of $\mathcal{G}_{_{n}w}$ is a triangular number gives rise to the partition of the vertices of $\mathcal{G}_{_{n}w}$  into subdivisions of cardinalities $n-1, n-2,\dots, 2,1$, so the total number of edges in $\mathcal{G}_{_{n}w}$  is  $2(n-2)+2(n-3)+\cdots + 4 +2$. Using the same argument, the total number of 4-cycles in $\mathcal{G}_{_{n}w}$ is $1+ 2 +\cdots + n-3.$  
\end{proof}
\begin{corollary}
The generating series for the vertex degree polynomials ${\rm P}_{\mathcal{G}_{_{n}w}}(d)$, $n\geq 4$ is given by
$$\sum_{n\geq 4}^{\infty}{\rm P}_{\mathcal{G}_{_{n}w}}(d)z^{n}=\frac{{\left(d^{4} z^{2} - 2 \, d^{3} z^{2} - d^{2} z^{3} + 2 \, d^{3} z + 3 \, d^{2} z^{2} + 2 \, d z^{3} - 4 \, d^{2} z - 4 \, d z^{2} + 2 \, d^{2} + 2 \, d z\right)} z^{3}}{{\left(1 -z\right)}^{3}}$$
\end{corollary}
\begin{remark}
The vertex-degree polynomial of $\mathcal{G}_{_{n}w}$ is closely connected with the Ehrhart polynomial of the $(n-2)$ dilation of the standard 2-simplex $\Delta_{2}$. In fact, this gives rise to  the existence of an isomorphism between the graph  $\mathcal{G}_{_{n}w}$ of reduced words of the permutation  ${_{n}w}$ and the lattice point graph $\mathcal{G}_{(n-2)\Delta_{2}}$ of the $(n-2)$-fold dilation of the standard 2-simplex.
\end{remark}
\section{The Hasse diagram of the Set $\mathcal{C}_{n-2,1,1}$}
\noindent The bijection between the set $R(_{n}w)$ of reduced words of the permutation $_{n}w$ and the set $\mathcal{C}_{n-2,1,1}$ of recording row strict tableaux induces an isomorphism between  the graph $\mathcal{G}_{_{n}w}$ of reduced words of the permutation $_{n}w$  and the Hasse diagram of the poset  $(\mathcal{C}_{n-2,1,1}, \leq)$ of recording row strict tableaux. We first describe the ordering on the recording row strict tableaux in $\mathcal{C}_{n-2,1,1}$  as follows.
\begin{defn}
Let $\tau_{1} = {\textup {\begin{tabular}{|c|c|c|c|}
\cline{1-4} $m_1$ & $m_{2}$   & . . .  &$m_{r-2}$ \\
\cline{1-4} $n_1$ \\
\cline{1-1}$n_2$ \\
\cline{1-1}\multicolumn{1}{c}{} \end{tabular}}}$ and $\tau_{2} = {\textup {\begin{tabular}{|c|c|c|c|}
\cline{1-4} $m'_1$ & $m'_{2}$   & . . .  &$m'_{r-2}$ \\
\cline{1-4} $n'_1$ \\
\cline{1-1}$n'_2$ \\
\cline{1-1}\multicolumn{1}{c}{} \end{tabular}}}$ be any two recording row-strict tableaux in $\mathcal{C}_{n-2,1,1}$. We say  $\tau_{1} \leq \tau_{2}$ if $m_{i} \geq m'_{i}$ for  $1\leq i\leq n-2$ and $n_{i} \leq n'_{i}$ for $1\leq i\leq 2$, otherwise they are incomparable. 
\end{defn}
\begin{defn}
We say $\tau_{2}$ covers $ \tau_{1}$ and write $\tau_{2}\succ\tau_{1}$ if there does not exist $\tau_{3}\in \mathcal{C}_{n-2, 1,1}$ such that $\tau_{1}<\tau_{3}<\tau_{2}$.
\end{defn}
\noindent Given a recording row strict tableau $\tau \in \mathcal{C}_{n-2,1,1}$, its row reading, denoted by $_{\tau}w$ is obtained by listing the row entries of $\tau$ starting  from bottom upward and going from left to right. For instance, the row reading of  $\tau_{1} = {\textup {\begin{tabular}{|c|c|c|}
\cline{1-3} 1 & 4  &  5 \\
\cline{1-3} 3 \\
\cline{1-1} 2\\
\cline{1-1}\multicolumn{1}{c}{} \end{tabular}}} \in \mathcal{C}_{3,1,1}$     is given by $_{\tau}w =23145$. Notice that  $_{\tau}w$ has a unique descent, therefore, it is Grassmannian. 
\begin{remark}
 The set $\mathcal{R}^{_{\tau}w}_{n}$  of row-readings of the recording row strict tableaux in $\mathcal{C}_{n-2, 1,1}$ constitutes the Grassmannian permutations whose associated partitions fit into the rectangle $n-2\times 2$. It is important to point out that these are precisely the indexing permutations of the Schubert varieties inside Grassmannian ${\rm Gr}(2,n)$ of planes in an $n$-dimensional complex vector space. In fact, There is a projection
 \begin{equation}
\pi : \mathcal{F}\ell _{n}(\C)\longrightarrow Gr(2,n)
\end{equation}
 from the full flag variety $\mathcal{F}\ell _{n}(\C)$ to the Grassmannian Gr(2,n) with $\pi^{-1}(X_{\lambda}(V_{\bullet})) = X_{w(\lambda)}(V_{\bullet})$, where $X_{\lambda}(V_{\bullet})$ is a Schubert variety in the Grassmannian $Gr(2,n)$. The permutation $w(\lambda)$  identified with the partition $\lambda=(\lambda_{1},\lambda_{2})$ is given by
 \begin{equation}
w_{i}=i+\lambda_{3-i}, \ 1\leq i\leq 2  \ {\rm and} \ w_{j}< w_{j+1},\ 3\leq j\leq n.
\end{equation}
\end{remark}
\begin{proposition}
Let $\tau_{1} = {\textup {\begin{tabular}{|c|c|c|c|}
\cline{1-4} $m_1$ & $m_{2}$   & . . .  &$m_{r-2}$ \\
\cline{1-4} $n_1$ \\
\cline{1-1}$n_2$ \\
\cline{1-1}\multicolumn{1}{c}{} \end{tabular}}}$ and $\tau_{2} = {\textup {\begin{tabular}{|c|c|c|c|}
\cline{1-4} $m'_1$ & $m'_{2}$   & . . .  &$m'_{r-2}$ \\
\cline{1-4} $n'_1$ \\
\cline{1-1}$n'_2$ \\
\cline{1-1}\multicolumn{1}{c}{} \end{tabular}}}\in \mathcal{C}_{n-2,1,1}$ be two recording row strict tableaux. Then $\tau_{2}$ covers $\tau_{1}$ if and only if $\ell(_{\tau_{2}}w) = \ell(_{\tau_{1}}w)+1$.
\end{proposition}
\begin{proof}
Suppose that $\tau_{1}$ covers $\tau_{2}$, by definition, we have  $\tau_{1} \leq \tau_{2}$, if $m_{i} \geq m'_{i}$ for  $1\leq i\leq n-2$ and $n_{i} \leq n'_{i}$ for $1\leq i\leq 2$ and  there does not exist $\tau_{3}\in \mathcal{C}_{n-2, 1,1}$ such that $\tau_{1}<\tau_{3}<\tau_{2}$ , so  either $n_{1}'=n_{1}$ and $n_{2}'=n_{2}+1$ or $n_{1}'=n_{1}+1$ and $n_{2}'=n_{2}$, therefore, $\ell(_{\tau_{2}}w) = \ell(_{\tau_{1}}w)+1$. On the other hand, consider the two row readings $_{\tau_{1}}w = n_{1}n_{2}m_{1}\cdots m_{n-2}$ and $_{\tau_{1}}w = n_{1}'n_{2}'m_{1}'\cdots m_{n-2}'$ in which $n_{1}<n_{2}>m_{1}<\cdots < m_{n-2}$ and $n_{1}'<n_{2}'> m_{1}'<\cdots < m_{n-2}'$  such that  $\ell(_{\tau_{2}}w) = \ell(_{\tau_{1}}w)+1$, so either $n_{1}'=n_{1}$ and $n_{2}'=n_{2}+1$ or $n_{1}'=n_{1}+1$ and $n_{2}'=n_{2}$, hence $\tau_{2}$ covers $\tau_{1}$.
\end{proof}
\noindent There is a natural partial order on the set  $\mathcal{R}^{_{\tau}w}_{n}$  of row-readings of the recording row strict tableaux in $\mathcal{C}_{n-2, 1,1}$ induced by the Bruhat order on the symmetric group $\mathfrak{S}_{n}$ described in section 2. Given any two row-readings $_{\tau_{1}}w, \ _{\tau_{2}}w\in\mathcal{R}^{_{\tau}w}_{n}$, we say $_{\tau_{1}}w\leq \  _{\tau_{2}}w$ if $\ell(_{\tau_{1}}w)\leq \ell( _{\tau_{2}}w)$. It is obvious that the Hasse diagram $\mathcal{H}_{\mathcal{R}^{_{\tau}w}_{n}}$ of the poset $(\mathcal{R}^{_{\tau}w}_{n}, \leq)$ is graded by the permutation length. The maximum element is the permutation $[n-1, n, 1, 2,\dots, n-2]$ while the minimum element is $[1,2,\dots, n]$ since every permutation in $\mathcal{R}^{_{\tau}w}_{n}$ has a unique descent at position 2.
\begin{example}
 The Hasse diagram of the poset $(\mathcal{R}^{_{\tau}w}_{4}, \leq)$ is given below
 \begin{figure}[!hbt]\label{fig_r=5}
          \begin{center}
	  \begin{tikzpicture}[x= 1.7cm, y=1.7cm]
	   \tikzset{vertex/.style = {shape=circle,draw,minimum size=10.5em}}
	    \tikzset{edge/.style = {-,> = latex'}}
                \node (a) at (-1,0.5) {${\color{red}1234}$};
    		\node (b) at (0,1) {${\color{red}1324}$};
		\node (c) at (1,1.5) {${\color{red}1423}$};
		\node (h) at (-1,2.5) {${\color{red}3412}$};
		\node (i) at (0,2) {${\color{red}2413}$};
		\node (j) at (-1,1.5) {${\color{red}2314}$};
		\draw (b)--(c);
		\draw (h)--(i);
		\draw (b)--(j)--(i)--(c); 
		\draw (a)--(b);
\end{tikzpicture}
	\end{center}
	\caption{The Hasse diagram of the poset  $(\mathcal{R}^{_{\tau}w}_{4}, \leq)$}		
\end{figure}
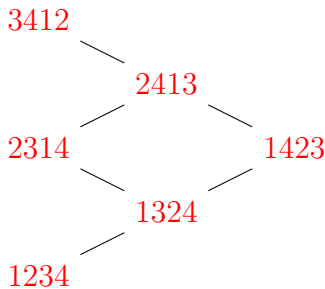
\end{example}
\begin{lemma}
 There is an order preserving bijective map between the posets  $( \mathcal{C}_{n-2, 1,1}, \leq)$  and  $(\mathcal{R}^{_{\tau}w}_{n}, \leq)$ for $n\geq 4$.
\end{lemma}
\begin{proof}
It is clear that the map $\tau\mapsto {_{\tau}w}_{n}$ is a bijection. We only need to show that it is order preserving. Suppose that $\tau_{1}\leq \tau_{2}$, by Definition 4.1, $m_{i} \geq m'_{i}$ for  $1\leq i\leq n-2$, \  $n_{i} \leq n'_{i}$ for $1\leq i\leq 2$, \  and $1\leq n_{i} , n'_{i}, m_{i} , m'_{i}\leq n$    so that $_{\tau_{1}}w = [ n_{1},n_{2},m_{1},\cdots , m_{n-2}]$ such that $n_{1}<n_{2}>m_{1}<\cdots < m_{n-2}$ and $_{\tau_{2}}w=[n_{1}',n_{2}', m_{1}',\cdots , m_{n-2}']$ such that $n_{1}'<n_{2}'> m_{1}'<\cdots < m_{n-2}'$. Notice that $\ell(_{\tau_{1}}w)\leq \ell( _{\tau_{2}}w)$, therefore, $_{\tau_{1}}w\leq \  _{\tau_{2}}w$.
\end{proof}
\begin{theorem}
The Hasse diagram $\mathcal{H}_{\mathcal{C}_{n-2, 1,1}}$ of the poset  $( \mathcal{C}_{n-2, 1,1}, \leq)$ is graded, with rank function defined by the permutation length given by
$$\rho(\tau) = \ell({_\tau}w)$$
for every $\tau\in \mathcal{C}_{n-2, 1,1}$ where ${_\tau}w$ is the row reading of $\tau$. Furthermore, $\mathcal{C}_{n-2, 1,1}$ is of the rank
$$\rho(\mathcal{C}_{n-2, 1,1})=2(n-2)$$
\end{theorem}
\begin{proof} It is straightforward to see that the graded structure of the Hasse diagram  $\mathcal{H}_{\mathcal{C}_{n-2, 1,1}}$ of $( \mathcal{C}_{n-2, 1,1}, \leq)$ is realised by identifying it with the Hasse diagram $\mathcal{H}_{\mathcal{R}^{_{\tau}w}_{n}}$ of its corresponding poset $(\mathcal{R}^{_{\tau}w}_{n}, \leq)$ of row-readings of the recording row-strict tableaux. The maximum element is the row-strict tableau $\tau'$ whose reading  $_{\tau'}w$ is the permutation $[n-1, n, 1, 2,\dots, n-2]$ while the permutation of the minimum element $\tau$ is $_{\tau}w=[1,2,\dots, n]$. The length $\ell(_{\tau'}w)$ of the permutation $_{\tau'}w$  associated to the maximum element $\tau'$ is $2(n-2)$, hence the rank of the poset.
\end{proof}

\section{Isomorphism between $\mathcal{G}_{_{n}w}$ and $\mathcal{G}_{(n-2)\Delta_{2}}$ }
\noindent  More is true about the chain of isomorphisms in the Corollary 4.8.  The identification of  the graph $\mathcal{G}_{_n{w}}$ with the Hasse diagrams $\mathcal{H}_{\mathcal{C}_{n-2, 1,1}}$ and $\mathcal{H}_{\mathcal{R}^{_{\tau}w}_{n}}$  reveals a fundamental connection between the graph of reduced words and the lattice point graph $\mathcal{G}_{(n-2)\Delta_{2}}$ of $(n-2)$-fold  dilation of the standard 2- simplex $\Delta_{2}$ . Recall that  the standard $2$-simplex $\Delta_2$ is the convex hull of the set $\{\underline{0}, e_1,e_2\}$ where $e_1, e_2$ are the standard vectors in $\R^{2}$ and $\underline{0}$ is the origin. That is,
\begin{equation}
\Delta_{2}:= \{{\bf x}\in\R ^{2} : {\bf x}\cdot e_{i}\geq 0, \ \ \sum_{i=1}^{2} {\bf x}\cdot e_{i}\leq 1\}
\end{equation}

\noindent and the dilation $k\Delta_{2}$, is  given by

\begin{equation}
k\Delta_{2}=\{{\bf x}\in\R ^{2} :{\bf x}\cdot e_{i}\geq 0, \ \ \sum_{i=1}^{2}{\bf x}\cdot e_{i}\leq k, \ \    k\in \N\}
\end{equation}

\noindent The standard 2- simplex $\Delta_{2}$ is a lattice polytope. Counting the lattice points of any given lattice polytope is the central theme of the Ehrhart theory  [3], [6], [10], [11] and [16]. In fact, the number of the lattice points on $k\Delta_2$ is given by the Ehrhart polynomial 
\begin{equation}
\mathcal{L}_{\Delta_{2}}(k)=|k\Delta_{2}\cap \Z_{\geq 0}^{2}|= {{k+2}\choose 2}.
\end{equation}
\noindent   Our interest is in the graph $\mathcal{G}_{k\Delta_{2}}$ of the set $\mathcal{C}:= k\Delta_{2}\cap \Z_{\geq 0}^{2}$ of lattice points on $k\Delta_{2}$. Notice that the elements of $\mathcal{C}$ are the integer solutions of the inequality 
\begin{equation}
\sum_{i=1}^{2}{\bf x}\cdot e_{i}\leq k
\end{equation}
 In order to construct the graph we define a lexicographical order $<_{\rm lex}$ on the elements of $\mathcal{C}$  as follows:  Let ${\bf a}=(a_1,a_2)$ and ${\bf b}=(b_1,b_2)$ be any two lattice points in $\mathcal{C}$. We say ${\bf a} <_{\rm lex} {\bf b}$ if in the integer coordinate difference  ${\bf a}-{\bf b}\in\Z^{2}$, the first coordinate is negative.  The order induces a directed graph $\stackrel{\rightarrow}{\mathcal{G}_{k\Delta_{2}}}$  whose vertices are given by the elements of the set $\mathcal{C}$ such that there is an edge  between ${\bf a} \  {\rm and}\ {\bf b}$ in $\mathcal{C}$ if  ${\bf b} \  {\rm covers}\ {\bf a}$ and write ${\bf a}<^{c}_{\rm lex}{\bf b}$, that is, there does not exist a lattice point ${\bf c}\in \mathcal{C}$ such that ${\bf a}<_{\rm lex} {\bf c}<_{\rm lex}{\bf b}$. The  sub-lattice graph $\mathcal{G}_{k\Delta_2}$ is the underlying graph of  the digraph $\stackrel{\rightarrow}{\mathcal{G}_{k\Delta_{2}}}$, that is,  the graph containing undirected edges in which all edge fibers are identical.  For instance, the sub-lattice graph  $\mathcal{G}_{3\Delta_2}$ is given in Figure 7 below.
\begin{figure}[!hbt]
	\begin{center}
	  \begin{tikzpicture}[x= 1.7cm, y=1.7cm]
	   \tikzset{vertex/.style = {shape=circle,draw,minimum size=10.5em}}
	    \tikzset{edge/.style = {-,> = latex'}}
		\node (a) at (0,2) {$(0,2)$};
		\node (b) at (1,1) {$(1,1)$};
		\node (c) at (2,0) {$(2,0)$};
		\node (d) at (1,0) {$(1,0)$};
		\node (e) at (0,0) {$(0,0)$};
		\node (f) at (0,1) {$(0,1)$};
		\node (g) at (0,3) {$(0,3)$};
		\node (h) at (1,2) {$(1,2)$};
		\node (i) at (2,1) {$(2,1)$};
		\node (k) at (3,0) {$(3,0)$};
		\draw (a)--(b)--(c)--(d)--(e);
		\draw (g)--(h)--(i)--(k)--(c);
		\draw (d)--(f)--(b);
		\draw (a)--(h);
		\draw (b)--(i);
		\end{tikzpicture}
	\end{center}
	\caption{ The  sub-lattice graph $\mathcal{G}_{3\Delta_2}$}
\end{figure}
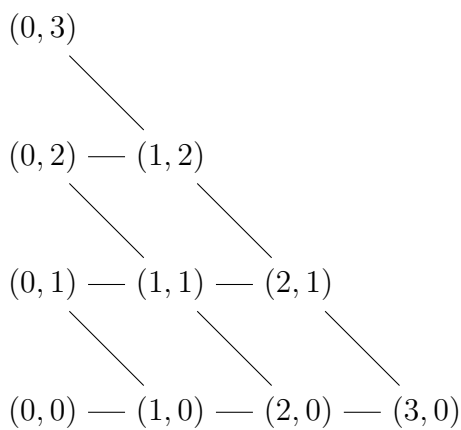
\noindent There is no edge between any pair of vertices of  $\mathcal{G}_{k\Delta_2}$ if they share the first coordinates. By setting $k$ to be $n-2$, it turns out that the vertices of the graph $\mathcal{G}_{k\Delta_2}$ of lattice points in $\mathcal{C}$ encode the fitted partitions associated with the Grassmannian permutations which constitute the vertices of the Hasse diagram Hasse diagram $\mathcal{H}_{\mathcal{R}^{_{\tau}w}_{n}}$ of the poset $(\mathcal{R}^{_{\tau}w}_{n}, \leq)$ described in the Remark 4.3. To see this we associate  to each lattice point ${\bf a}=(a_1,a_2)\in \mathcal{C}$, a weight ${\bf m}_{\bf a}$  given by 
\begin{equation}
{\bf m}_{\bf a} = \sum_{t=1}^{2} ta_{t}
\end{equation}
and partition $\lambda = (\lambda_1, \lambda_2)$ such that $\lambda_{1} = a_{1}+a_{2}$ and $\lambda_{2} =a_{2}$
\begin{proposition}
Let ${\bf a}=(a_1,a_2)$ be a lattice point in $\mathcal{C}$ such that  ${\bf m}_{\bf a}$ and $\lambda=(\lambda_1,\lambda_2)$ are respectively the weight and the partition associated with ${\bf a}$. Then $\lambda$ is a partition of ${\bf m}_{\bf a}$ which fits into the rectangle $\square_{k\times 2}$.
\end{proposition}
\begin{proof}
$\lambda_1 + \lambda_2 = (1,2)\cdot{\bf a} = {\bf m}_{\bf a}$ for a fixed vector $(1,2)$. Since the maximum lattice point in $\mathcal{C}$, we have  $0\leq \lambda_i\leq k$,  for  $1\leq i\leq 2$, hence $\lambda$ fits into the rectangle $\square_{k\times 2}$.
\end{proof}
\noindent We now give the poset associated to the fitted partition. This is well known in Schubert calculus of Grassmannians. Let $\mathcal{C}_{\square_{k\times 2}}$ be the set of fitted partitions into the rectangle $\square_{k\times 2}$, The partial order on $\mathcal{C}_{\square_{k\times 2}}$ is well known [], for any $\lambda, \mu\in \mathcal{C}_{\square_{k\times 2}}$, we say $\lambda\leq\mu$ if  $0\leq\lambda_{i}\leq\mu_{i}\leq k$, for $1\leq i\leq 2$. This gives rise to the Young's lattice $\mathcal{Y}_{\mathcal{C}_{\square_{k\times 2}}}$ on the corresponding Young diagrams. It is graded since the rank function is given by the number of boxes in each of the diagrams. For instance, the Young's lattice of $\mathcal{C}_{\square_{3\times 2}}$ is given in the Figure 7.

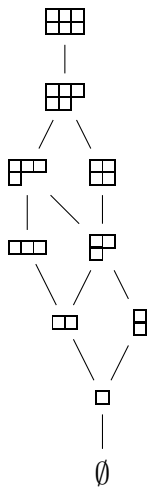
\begin{figure}[!hbt]
               \begin{center}
		\begin{tikzpicture}
		\node (a) at (0,0) {$\emptyset$};
		\node (b) at (0,1) {$\Yboxdim{5pt}\tiny\yng(1)$};
		\node (c) at (0.5,2) {$\Yboxdim{5pt}\tiny\yng(1,1)$};
		\node (d) at (-0.5,2) {$\Yboxdim{5pt}\tiny\yng(2)$};
		\node (e) at (0,3) {$\Yboxdim{5pt}\tiny\yng(2,1)$};
		\node (g) at (0,4) {$\Yboxdim{5pt}\tiny\yng(2,2)$};
		\node (f) at (-1,3) {$\Yboxdim{5pt}\tiny\yng(3)$};
		\node (h) at (-1,4) {$\Yboxdim{5pt}\tiny\yng(3,1)$};
		\node (i) at (-0.5,5) {$\Yboxdim{5pt}\tiny\yng(3,2)$};
		\node (j) at (-0.5,6) {$\Yboxdim{5pt}\tiny\yng(3,3)$};
		\draw (a)--(b)--(c)--(e)--(g)--(i)--(j);
		\draw (b)--(d)--(f)--(h)--(i);
		\draw (d)--(e);
		\draw (e)--(h);
	\end{tikzpicture}
	\end{center}
	\caption{The Young's lattice  of $\mathcal{C}_{\square_{3\times 2}}$.}	
\end{figure}

\noindent Now, for a fixed vector  ${\bf z}=(1,2)$, consider the set 
\[ N^{\bf z}_{{\bf m}_{\bf a}} = \#\{ {\bf a} \in k\Delta_{2}\cap \Z_{\geq 0}^{2} : {\bf a}\cdot {\bf z} = {\bf m}_{\bf a}  \}.\]
Geometrically, it means the simplex $k\Delta_{2}$ is being sliced by the line segments $a_1 +2a_2 = {\bf m}_{\bf a}$ and the lattice points of  $k\Delta_{2}$ on each of these lines are being counted. The reader is referred to [1] and [2].  Notice that $0\leq {\bf m}_{\bf a}\leq 2k$ so that the graded polynomial $ {\rm P}^{\bf z}_{k}(q)$ is given by 
\begin{equation}
 {\rm P}^{\bf z}_{k}(q)= \sum_{{\bf m}_{\bf a}=0 }^{2k} N^{\bf z}_{{\bf m}_{\bf a}} q^{{\bf m}_{\bf a}}= \left[ k+2 \choose 2\right]_q
\end{equation}
with the product expansions given given by 
\begin{equation}
{\rm G}(q,t) = \prod^{2}_{i=0}\frac{1}{(1-q^{i}t)}.
\end{equation}
The graded polynomial ${\rm P}^{\bf z}_{k}(q)$ is a refinement of the Ehrhart polynomial $\mathcal{L}_{\Delta_{2}}(k)$ of the $k$-fold dilation of the standard 2-simplex. Therefore, the Hasse diagram $\mathcal{G}_{k\Delta_2}$ of the poset $(k\Delta_{2}\cap \Z_{\geq 0}^{2}, <_{\rm lex})$ is graded with rank function $\rho({\bf a})={\bf m}_{\bf a}$ where $(0, k)$ and $(0,0)$ are respectively the maximum and minimum elements. In fact, its rank polynomial is  the equation $(5.6)$. This is deeply connected with the Poincar\'e polynomial of the cohomology ${\rm H}^{\ast}({\rm Gr}(2,n), \Z)$ of the Grassmannian of 2-planes in the 4-complex space, that is, the Hilbert series for a graded ring arising from the Borel presentation of the cohomology ring. It is well known that to every fitted partition $\lambda\subseteq\square_{k\times2}$ there is a corresponding Grassmannian permutation $w(\lambda)$ given in the equation $(4.1)$. These permutations, denoted by $\mathcal{R}^{_{\tau}w}_{n}$, are precisely the row readings of the set $\mathcal{C}_{n-2,1,1}$ of recording row-strict tableaux $\tau$ in which the tableaux encode  the positions of descents and ascents of the vertices of the graph $\mathcal{G}_{_{n}w}$ of reduced words of the permutation $_{n}w$.
\begin{proposition}
Let ${\bf a}=(a_1,a_2)$ be a lattice point in $\mathcal{C}$ such that  ${\bf m}_{\bf a}$ and $\lambda=(\lambda_1,\lambda_2)$ are respectively the weight and the partition associated with ${\bf a}$. Then the length of the Grassmannian permutation $w(\lambda)\in\mathcal{R}^{_{\tau}w}_{n}$ corresponding to $\lambda$ is  ${\bf m}_{\bf a}$.
\end{proposition}
\begin{proof}
Since the Grasmmannian permutation $w(\lambda)$ is of the form $$w_{i}=i+\lambda_{3-i}, \ 1\leq i\leq 2  \ {\rm and} \ w_{j}< w_{j+1},\ 3\leq j\leq n,$$ so its Lehma code $c(w(\lambda))$ is given by $(w_{1}-1, w_{2}-2, 0,\dots,0)$. The sum of the entries of the code is the length of the permutation since each entry in the $i^{th}$ coordinate by definition denotes the number of inversions associated with the value $w_{i}$ of the permutation. The non-increasing rearrangement of nonzero entries of the code $c(w(\lambda))$ is precisely the partition $\lambda$, hence, $$w_{2} +w_{1}-3 =\ell(w(\lambda)) = \lambda_1 +\lambda_2 = {\bf m}_{\bf a}$$
\end{proof}
\noindent The projection  $\pi$ in the equation $(4.1)$ induces a monomorphism $\pi^{\ast}$ at the level of cohomology.
\begin{equation} 
\pi^{\ast} :{\rm H}^{\ast} ({\rm Gr}(2,n), \Z) \longrightarrow {\rm H}^{\ast}(\mathcal{F}\ell _{n}(\C),\Z) 
\end{equation}  
which takes cycle $\sigma_{\lambda}$ to  the cycle $\sigma_{w(\lambda)}$. The reader is referred to [6] and [8] .The cohomology ring of the Grassmannian $Gr(2,n)$ is generated by the Schubert cycles $\sigma_{\lambda}$. These are Poincar\'e dual of the fundamental classes in the homology of Schubert varieties. The ring is isomorphic to the coinvariant algebra via Borel presentation, that is,
\begin{equation}
{\rm H}^{\ast} ({\rm Gr}(2,n), \Z)\cong \frac{\Z[x_1,\dots, x_n]^{\frak{S}_{2}\times\frak{S}_{n-2}}}{\Z^{+}[x_1,\dots,x_n]^{\frak{S}_{n}}}.
\end{equation}
Therefore, the Hilbert series of the coinvariant algebra  is the Poincar\'e polynomial of the cohomology ring, that is,   the Gaussian polynomial $f(t)$ given by 
\begin{equation}
f(t)=\frac{(1-t^{n})(1-t^{n-1})}{(1-t)(1-t^2)}.
\end{equation}
Notice that this is precisely the rank polynomial  given in $(5.6)$. 
\begin{proposition}
Setting $k=n-2$, there is an order preserving between the posets $(k\Delta_{2}\cap \Z_{\geq 0}^{2}, <_{\rm lex})$ and $(\mathcal{R}^{_{\tau}w}_{n}, \leq)$ making the Hasse diagrams $\mathcal{G}_{k\Delta_2}$ and  $\mathcal{H}_{\mathcal{R}^{_{\tau}w}_{n}}$ isomorphic.
\end{proposition}
\begin{corollary}
The following isomorphisms hold
$$\mathcal{G}_{_{n}w}\cong \mathcal{H}_{\mathcal{C}_{n-2, 1,1}}\cong \mathcal{H}_{\mathcal{R}^{_{\tau}w}_{n}}\cong \mathcal{Y}_{\mathcal{C}_{\square_{k\times 2}}}\cong\mathcal{G}_{k\Delta_2} $$
\end{corollary}
\begin{remark}
Notice that the number of vertices of $\mathcal{G}_{_{n}w}$ which contain braid is the rank of the Hasse diagrams $\mathcal{H}_{\mathcal{C}_{n-2, 1,1}}$, $\mathcal{H}_{\mathcal{R}^{_{\tau}w}_{n}}$, $\mathcal{Y}_{\mathcal{C}_{\square_{k\times 2}}}$and $\mathcal{G}_{k\Delta_2}$
\end{remark}

\begin{example}
Consider the set $\mathcal{C}:= 3\Delta_{2}\cap \Z_{\geq 0}^{2}$ of lattice points on the $3$-fold dilation of the standard 2-simplex $\Delta_2$. That is,
$$\mathcal{C}=\{ (0,3), (1,2), (0,2), (2,1), (1,1), (3,0), (0,1), (2,0), (1,0), (0,0) \}.$$ The set $\mathcal{C}_{\lambda}$ of the corresponding fitted partitions $\lambda\subseteq\square_{3\times 2}$ is given by
$$\mathcal{C}_{\lambda}= \{ (3,3), (3,2), (2,2), (3,1), (2,1]), (3), (1,1), (2), (1), \emptyset \}$$ , so  the set  $\mathcal{C}_{w(\lambda)}$ of the corresponding Grassmannian permutations is given by
$$\mathcal{C}_{w(\lambda)}=\{45123, 35124, 34125, 25134, 24135,  15234, 23145, 14235, 13245, 12345\}=\mathcal{R}^{_{\tau}w}_{5}$$
\[\mathcal{C}_{3,1,1}=\left\lbrace \begin{array}{l} \vspace{10pt} \vcenter{\hbox{\young(123,5,4),\, \young(124,5,3),\, \young(125,4,3),\, \young(134,5,2),\, \young(135,4,2),\, \young(234,5,1)}} \\ \vcenter{\hbox{\young(145,3,2),\, \young(235,4,1),\, \young(245,3,1),\, \young(345,2,1) }}\end{array}\right\rbrace .\] 
$$\mathcal{G}_{_{n}w}= \{ 432134, 432314, 432341, 423214, 423241, 243214, 423421,  243241, 243421, 234321\}$$
\end{example}
 \noindent {\bf Acknowledgment:} I would like to thank Balazs Szendroi; Ben Young and Mike Zabrocki for  productive discussions during the preparation of the manuscript. 

 \begin{center}
{\bf References}
\end{center}

\begin{enumerate}
\item[{[1]}] P. Adeyemo and B. Szendr\"{o}i. \emph{ Refined Ehrhart Series and Bigraded rings}. Studia Scientiarum Mathemacarum Hungarica. Volume 60, Issue 2-3(2023), p. 97-108.
\item[{[2]}] P. Adeyemo. \emph{Grassmannians in the Lattices Points of Dilations of the Standard Simplex}. International J. Math. Combin. Vol 1. (2023), 1-20.
\item [{[3.]}] S. Assaf and A. Schilling, \emph{A Demazure crystal construction for Schubert polynomial} Volume 1, issue 2 (2018), p. 225-247.
\item[{[4.]}]  A. Bj\"{o}rner and F. Brenti, \emph{Combinatorics of Coxeter Groups}, Graduate Texts in Mathematics. Springer. 2005.
\item[{[5.]}] P. Edelman and C. Greene,\emph{Balanced tableaux,} Adv. in Math. 63(1987) no. 1 42-99.
\item[{[6.]}] W. Fulton, \emph{Young tableaux}, volume 35 of London Mathematical Society Student Texts. Cambridge University Press, Cambridge, 1997.
\item [{[7.]}]T. Lam, \emph{Affine Stanley symmetric functions}, Amer. J. Math. 128 (2006), no. 6, 1553-1586.
\item[{[8.]}] V. Lakshhmibai and J. Brown. \emph{Flag Varieties: An Interplay of Geometry, Combinatorics, and Representation Theory}. Texts and Reading in Mathematics, Vol. 59.
\item[{[9.]}]A. Lascoux and M. P. Sch\"{u}tzenberger, \emph{Structure de Hopf de l'anneau de cohomologie et
de l'anneau de Grothendieck d'une variet$\acute{e}$ de drapeaux, }C.R. Acad. Sci. Paris 295  (1982.), 629-633.
\item[{[10.]}] B.E Sagan, \emph{The Symmetric Group: Representations, Combinatorial Algorithms, and Symmetric Functions}, (2nd Ed.), (Springer, 2013).
\item[{[11]}]  Schilling, A., N.M. Thi?ery, G. White and N. Williams, \emph{Braid moves in commutation classes}
of the symmetric group?, European J. Combin., 62, 15?34, (2017).
\item [{[12.]}]R. Stanley, \emph{On the number of reduced decompositions of elements of Coxeter groups}, European J. Combinatorics 5(1984)
   359--509.  
\item[{[13]}]  Stanley, R.P., Enumerative Combinatorics, Vol. 1 (2nd Ed.), (Cambridge University Press,
Cambridge, 2012).    
\item[{[14]}] Tenner, B., \emph{On the expected number of commutations in reduced words}, Australas. J.
Combin. 62, 147?154 (2015).

\end{enumerate}

\end{document}